\documentclass[12pt]{article}

\usepackage[a4paper,left=0.82in,right=0.82in,top=0.95in,bottom=0.95in]{geometry}
\usepackage[T1]{fontenc}
\usepackage[utf8]{inputenc}
\usepackage{lmodern}
\usepackage{amsmath,amssymb,amsthm,mathtools}
\usepackage{booktabs}
\usepackage{graphicx}
\usepackage{xcolor}
\usepackage{hyperref}
\usepackage{microtype}
\usepackage{enumitem}
\usepackage{array}
\usepackage{setspace}

\hypersetup{
  colorlinks=true,
  linkcolor=blue!55!black,
  citecolor=blue!55!black,
  urlcolor=blue!55!black
}

\setlist[itemize]{topsep=0.25em,itemsep=0.12em}
\setlist[enumerate]{topsep=0.25em,itemsep=0.12em}

\theoremstyle{plain}
\newtheorem{theorem}{Theorem}
\newtheorem{proposition}[theorem]{Proposition}
\newtheorem{lemma}[theorem]{Lemma}

\theoremstyle{definition}
\newtheorem{definition}[theorem]{Definition}
\theoremstyle{remark}
\newtheorem{remark}[theorem]{Remark}

\newcommand{\Lop}{\mathcal L}
\newcommand{\floor}[1]{\left\lfloor #1\right\rfloor}
\newcommand{\ceil}[1]{\left\lceil #1\right\rceil}

\begin{document}

\title{%
\includegraphics[width=0.33\linewidth]{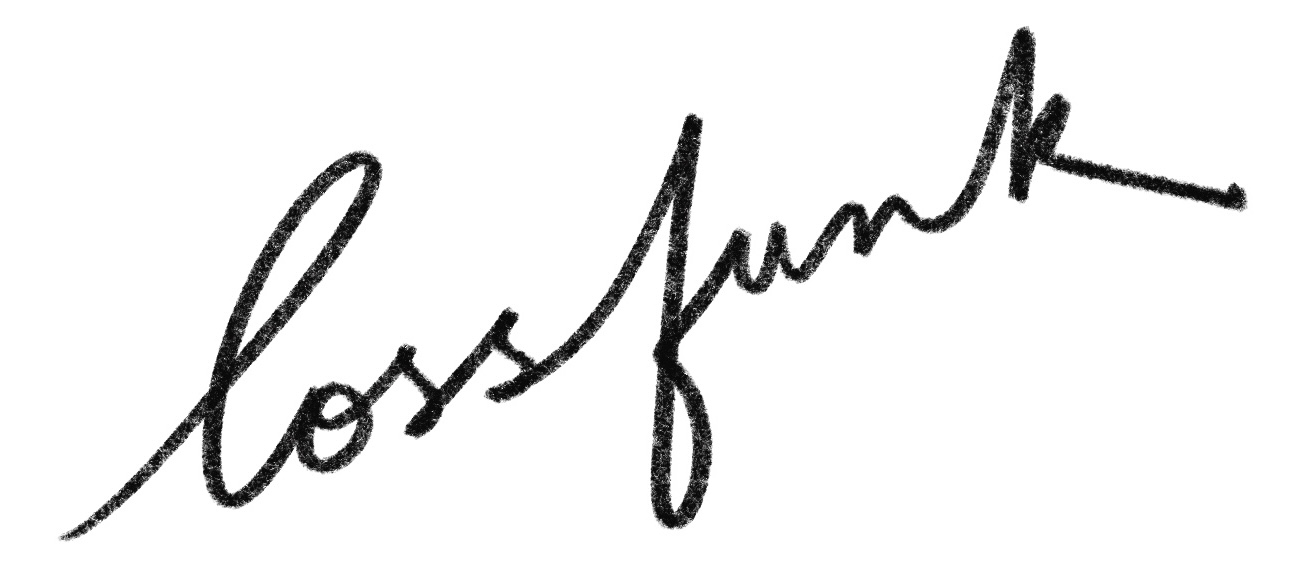}\\[0.85cm]
A Two-Graph Refinement of Paulsen's Lollipop Bounds
}

\author{Siddhartha Mahajan \and Paras Chopra}
\date{}
\maketitle

\begin{abstract}
Let \(a_{\Lop}(n)\) be the maximum number of regions into which \(n\)
lollipops divide the plane.  Paulsen introduced a second obstruction for this
problem, based on pairs of circles meeting at obtuse angle, in addition to the
stem-direction obstruction of Cutler--Karlsson--Sloane.  We recast Paulsen's
argument as a weighted problem for two graphs: a \(K_4\)-free graph \(D\) of
non-close stem pairs and a \(K_5\)-free graph \(E\) of non-intriguing circle
pairs.  For the total number \(C\) of pairwise crossings,
\[
  C\le 4\binom n2+|D|+|E|+|D\cap E|.
\]
Paulsen bounds the final term by \(|D|\).  We keep the overlap term and analyze
near-extremal configurations of \(D\) and \(E\).  This closes all of Paulsen's
remaining gaps up to \(n=17\), and also closes \(n=19\):
\[
\begin{gathered}
  a_{\Lop}(0),a_{\Lop}(1),\ldots,a_{\Lop}(17)\\
  =1,2,10,25,45,71,104,142,186,237,294,356,425,500,580,667,761,859,
\end{gathered}
\]
and
\[
  a_{\Lop}(19)=1076.
\]
The same method gives the one-region gaps
\[
  964\le a_{\Lop}(18)\le965,
  \qquad
  1193\le a_{\Lop}(20)\le1194.
\]
\end{abstract}

\section{Introduction and statement}

A \emph{lollipop} is a circle together with a half-line attached at one point
of the circle and extending radially outward.  Equivalently, it is determined
by a center \(C\in\mathbb R^2\), a radius \(r>0\), and a unit vector \(u\), and
is the union of
\[
  \{P: |P-C|=r\}
  \quad\text{and}\quad
  \{C+ru+tu:t\ge0\}.
\]
Let \(a_{\Lop}(n)\) be the largest possible number of connected regions into
which \(n\) lollipops divide the plane.

We work in generic arrangements: there are no tangencies, no triple crossing
points, and no crossing occurs at an anchor.  This entails no loss for an
extremal problem, since a non-generic maximizer can be perturbed without
decreasing the number of regions.  If \(C\) is the total number of crossings
between distinct lollipops, then
\begin{equation}\label{eq:regions-crossings}
  \#\text{regions}=C+n+1.
\end{equation}
Thus the region problem is equivalent to maximizing \(C\).

The values \(a_{\Lop}(0),\ldots,a_{\Lop}(4)=1,2,10,25,45\) are due to
Cutler, Karlsson, and Sloane \cite{CKS}.  Paulsen's recent note proves the
exact values for \(n=5,6,7,10\) and gives the best previously published upper
bounds for the remaining small cases \cite{Paulsen}.  The OEIS entry A389624
has been updated accordingly \cite{OEIS}.

Our contribution is purely combinatorial once Paulsen's two geometric inputs
are accepted.  We do not change the known construction; instead we sharpen the
upper bound by analyzing how the two extremal obstruction graphs can overlap.

\begin{theorem}\label{thm:main}
The following values are exact:
\[
\begin{aligned}
&a_{\Lop}(0),a_{\Lop}(1),\ldots,a_{\Lop}(17)\\
&\qquad=1,2,10,25,45,71,104,142,186,237,294,356,425,500,580,667,761,859,
\end{aligned}
\]
and
\[
  a_{\Lop}(19)=1076.
\]
Moreover,
\[
  964\le a_{\Lop}(18)\le965,
  \qquad
  1193\le a_{\Lop}(20)\le1194.
\]
\end{theorem}

The matching lower bounds in Theorem \ref{thm:main} come from the standard
blow-up of Karlsson's optimal four-lollipop configuration.  We recall the
construction in Appendix \ref{app:lower-table}.  The rest of the paper proves
the upper bounds.

\section{Paulsen's two obstructions}

We use Paulsen's terminology \cite{Paulsen}.

\begin{definition}
A pair of lollipops is \emph{close} if the angle between its two stem
directions is at most \(90^\circ\).  A pair is \emph{intriguing} if its two
circles either do not intersect, or intersect at angle at most \(90^\circ\),
where the intersection angle is measured after orienting both circles in the
same direction.
\end{definition}

The following are the geometric inputs.  The first close-pair estimate is the
stem-angle lemma of Cutler--Karlsson--Sloane, and the intriguing-pair estimates
and forced intriguing-pair theorem are due to Paulsen.

\begin{lemma}[Cutler--Karlsson--Sloane and Paulsen]\label{lem:paulsen-inputs}
For a generic pair of lollipops:
\begin{enumerate}[label=(\roman*)]
  \item a close pair has at most \(5\) crossings;
  \item an intriguing pair has at most \(5\) crossings;
  \item a pair that is both close and intriguing has at most \(4\) crossings.
\end{enumerate}
Furthermore, among any four lollipops there is a close pair, and among any
five lollipops there is an intriguing pair.
\end{lemma}

For reference, Paulsen's linear-algebra proof of the last assertion is recalled
in Appendix \ref{app:paulsen-proof}, with a minor sign-cleanup to handle zero
coefficients.

Given an arrangement of \(n\) lollipops, define two graphs \(D\) and \(E\) on
\([n]\):
\[
  ij\in D
  \quad\Longleftrightarrow\quad
  L_i,L_j\text{ are not close},
\]
\[
  ij\in E
  \quad\Longleftrightarrow\quad
  L_i,L_j\text{ are not intriguing}.
\]
Lemma \ref{lem:paulsen-inputs} implies
\[
  D\text{ is }K_4\text{-free},
  \qquad
  E\text{ is }K_5\text{-free}.
\]

Let \(N=\binom n2\), and let \(c_{ij}\) be the number of crossings between
\(L_i\) and \(L_j\).  The four pair estimates in Lemma
\ref{lem:paulsen-inputs} are summarized by
\begin{equation}\label{eq:pair-weight}
  c_{ij}\le
  4+\mathbf{1}_{ij\in D}
   +\mathbf{1}_{ij\in E}
   +\mathbf{1}_{ij\in D\cap E}.
\end{equation}
Summing over all pairs gives
\begin{equation}\label{eq:two-graph-bound}
  C\le 4N+\sigma(D,E),
  \qquad
  \sigma(D,E):=|D|+|E|+|D\cap E|.
\end{equation}

Let \(t_r(n)\) denote the number of edges of the balanced complete
\(r\)-partite Turan graph \(T_r(n)\).  Since \(D\) is \(K_4\)-free and \(E\)
is \(K_5\)-free,
\[
  |D|\le t_3(n),
  \qquad
  |E|\le t_4(n).
\]
Paulsen additionally uses \(|D\cap E|\le |D|\) and obtains
\begin{equation}\label{eq:paulsen-crossing-upper}
  C\le U_P(n):=4\binom n2+2t_3(n)+t_4(n).
\end{equation}
Equivalently,
\[
  C\le
  7\binom n2
  -2\ceil{\frac{n(n-3)}6}
  -\ceil{\frac{n(n-4)}8}.
\]
We improve \eqref{eq:paulsen-crossing-upper} by showing that equality, and in
many cases near equality, is incompatible with the simultaneous structures of
\(D\) and \(E\).

\section{One-defect cases}

The easiest improvements use only the equality case of Turan's theorem.

\begin{proposition}\label{prop:one-defect}
For \(n\in\{8,9,11,13\}\),
\[
  C\le U_P(n)-1.
\]
Consequently,
\[
  a_{\Lop}(8)=186,
  \quad
  a_{\Lop}(9)=237,
  \quad
  a_{\Lop}(11)=356,
  \quad
  a_{\Lop}(13)=500.
\]
\end{proposition}

\begin{proof}
Equality in \eqref{eq:paulsen-crossing-upper} would force
\[
  |D|=t_3(n),
  \qquad
  |E|=t_4(n),
  \qquad
  D\subseteq E.
\]
Thus \(D=T_3(n)\) and \(E=T_4(n)\).  Since \(D\subseteq E\), every part of the
complete quadripartite graph \(E\) must be contained in one part of the
complete tripartite graph \(D\).  Equivalently, the part sizes of \(T_4(n)\)
must be groupable into the part sizes of \(T_3(n)\).

For the four values in question this grouping is impossible:
\[
\begin{array}{@{}c c c@{}}
\toprule
n&T_3(n)\text{ part sizes}&T_4(n)\text{ part sizes}\\
\midrule
8&3,3,2&2,2,2,2\\
9&3,3,3&3,2,2,2\\
11&4,4,3&3,3,3,2\\
13&5,4,4&4,3,3,3\\
\bottomrule
\end{array}
\]
Therefore \(\sigma(D,E)\le 2t_3(n)+t_4(n)-1\), so
\(C\le U_P(n)-1\).  The construction in Appendix \ref{app:lower-table}
attains the resulting upper bound.
\end{proof}

\section{Internal edges over an extremal tripartition}

The remaining improvements need a small extremal lemma.  Assume first that
\(D=T_3(n)\), with parts \(A,B,C\) of sizes
\[
  p\ge q\ge r.
\]
Write
\[
  m=|D\setminus E|,
  \qquad
  x=|E\setminus D|.
\]
Thus \(m\) counts missing cross-edges of \(E\) relative to the tripartition,
while \(x\) counts internal \(E\)-edges inside \(A\), \(B\), and \(C\).

\begin{lemma}\label{lem:internal-bounds}
Assume \(p\ge q\ge r\ge4\), \(D=T_3(n)\), and \(E\) is \(K_5\)-free.  With
\(m\) and \(x\) as above:
\begin{enumerate}[label=(\alph*)]
  \item if \(m=0\), then
  \[
    x\le\floor{p^2/4};
  \]
  \item if \(m=1\), then
  \[
    x\le\max\left\{\floor{p^2/4},\,p+q-2\right\};
  \]
  \item if \(m=2\), then
  \[
    x\le\max\left\{\floor{p^2/4},\,p+q-2,\,2p+q-5\right\}.
  \]
\end{enumerate}
\end{lemma}

\begin{proof}
Let \(H_A,H_B,H_C\) be the internal graphs induced by \(E\) on the three parts.
We call a part \emph{active} if its internal graph has at least one edge.

First, each \(H_A,H_B,H_C\) is triangle-free.  Indeed, if, say, \(H_A\)
contained a triangle, then since \(m\le2\) and \(|B|,|C|\ge4\), we could choose
vertices \(b\in B\) and \(c\in C\) incident to no missing cross-edge.  The
triangle together with \(b\) and \(c\) would span a \(K_5\) in \(E\), a
contradiction.

Second, suppose \(X\) and \(Y\) are two active parts, with third part \(Z\).
For every internal edge \(e\in H_X\) and every internal edge \(f\in H_Y\), some
missing cross-edge must have one endpoint in \(e\) and one endpoint in \(f\).
Otherwise, choosing a vertex of \(Z\) incident to no missing cross-edge would
make \(e\cup f\cup\{z\}\) a \(K_5\) in \(E\).

If \(m=0\), the second observation implies that at most one part is active.
Mantel's theorem inside the largest active part gives \(x\le\floor{p^2/4}\).

Let \(m=1\).  If only one part is active, the same Mantel bound applies.  If
two parts \(X,Y\) are active, the unique missing cross-edge must join \(X\) to
\(Y\), say \(xy\).  The second observation implies that every edge of \(H_X\)
contains \(x\), and every edge of \(H_Y\) contains \(y\).  Hence the two
internal graphs are stars, contributing at most \((|X|-1)+(|Y|-1)\le p+q-2\)
edges.

Let \(m=2\).  One active part gives the Mantel bound.  Three active parts are
impossible, since every pair of active parts would require a missing cross-edge
between them.  Thus suppose exactly two parts \(X,Y\) are active, with
\(|X|=s\ge t=|Y|\).  If only one of the two missing cross-edges joins \(X\) to
\(Y\), the preceding star argument gives at most \(s+t-2\le p+q-2\) internal
edges.

It remains to handle the case where both missing cross-edges join \(X\) to
\(Y\).  Write them as \(a_1b_1\) and \(a_2b_2\), where \(a_i\in X\) and
\(b_i\in Y\), allowing the two missing edges to share an endpoint.  The second
observation says that for every \(e\in H_X\) and \(f\in H_Y\), there is an
\(i\in\{1,2\}\) with \(a_i\in e\) and \(b_i\in f\).

If the two missing edges share an endpoint in \(X\), then every edge of
\(H_X\) contains that endpoint, while every edge of \(H_Y\) meets
\(\{b_1,b_2\}\).  Since \(H_Y\) is triangle-free, the latter contributes at
most \(2t-4\) edges, and \(H_X\) contributes at most \(s-1\).  The total is at
most \(s+2t-5\le2s+t-5\).  The case of a shared endpoint in \(Y\) is even
smaller, by symmetry and \(s\ge t\).

Finally assume the two missing edges are disjoint.  Every edge of \(H_X\) meets
\(\{a_1,a_2\}\), and every edge of \(H_Y\) meets \(\{b_1,b_2\}\).  Classify an
edge by which distinguished endpoints it contains.  If \(H_X\) contains both
an edge using only \(a_1\) and an edge using only \(a_2\), then every edge of
\(H_Y\) must contain both \(b_1\) and \(b_2\), so \(|E(H_Y)|\le1\) and the
triangle-free bound \(|E(H_X)|\le2s-4\) gives total at most \(2s-3\le2s+t-5\).
Otherwise all edges of \(H_X\) contain a common distinguished endpoint, so
\(|E(H_X)|\le s-1\); the same classification gives \(|E(H_Y)|\le2t-4\), and
the total is at most \(s+2t-5\le2s+t-5\).  This proves the \(m=2\) bound.
\end{proof}

We also need one standard covering fact.

\begin{lemma}\label{lem:k4-cover}
Let \(K_{a,b,c,d}\) be a complete quadripartite graph with
\(a\le b\le c\le d\).  To destroy all transversal \(K_4\)'s, one must delete
at least \(ab\) edges.  This is sharp.
\end{lemma}

\begin{proof}
Deleting all edges between the two smallest parts destroys every transversal
\(K_4\), so \(ab\) deletions suffice.

For the lower bound, assign weight \(1/(cd)\) to every transversal \(K_4\).
The total weight is
\[
  abcd\cdot\frac1{cd}=ab.
\]
An edge between parts of sizes \(s\) and \(t\) lies in the product of the other
two part sizes many transversal \(K_4\)'s.  This product is at most \(cd\), so
each deleted edge covers weight at most \(1\).  At least \(ab\) deleted edges
are necessary.
\end{proof}

\section{Two- and three-defect cases}

The next propositions apply the preceding lemmas to \(\sigma(D,E)\).

\begin{proposition}\label{prop:two-defect}
For \(n\in\{12,14,16\}\),
\[
  C\le U_P(n)-2.
\]
Consequently,
\[
  a_{\Lop}(12)=425,
  \qquad
  a_{\Lop}(14)=580,
  \qquad
  a_{\Lop}(16)=761.
\]
\end{proposition}

\begin{proof}
If \(|D|\le t_3(n)-1\), then
\[
  \sigma(D,E)\le2|D|+|E|\le2(t_3(n)-1)+t_4(n),
\]
so two crossings are already saved.

It remains to consider \(D=T_3(n)\).  Put \(\delta=t_4(n)-t_3(n)\).  Since
\[
  |E|=t_3(n)-m+x,
  \qquad
  |D\cap E|=t_3(n)-m,
\]
we have
\begin{equation}\label{eq:sigma-tripartite}
  \sigma(D,E)=3t_3(n)+x-2m.
\end{equation}
Saving two crossings is equivalent to
\[
  x-2m\le\delta-2.
\]
If \(m\ge2\), this follows from \(|E|\le t_4(n)\), since \(x\le\delta+m\) and
therefore \(x-2m\le\delta-m\le\delta-2\).  For \(m=0,1\), Lemma
\ref{lem:internal-bounds} gives the following bounds:
\[
\begin{array}{@{}c c c c c@{}}
\toprule
n&(p,q,r)&\delta&m=0&m=1\\
\midrule
12&(4,4,4)&6&4&6\\
14&(5,5,4)&8&6&8\\
16&(6,5,5)&11&9&9\\
\bottomrule
\end{array}
\]
In each row these imply \(x\le\delta-2\) for \(m=0\) and \(x\le\delta\) for
\(m=1\), as required.  The lower construction attains the resulting upper
bound.
\end{proof}

\begin{proposition}\label{prop:three-defect}
For \(n\in\{15,17,18,19,20\}\),
\[
  C\le U_P(n)-3.
\]
Consequently,
\[
  a_{\Lop}(15)=667,
  \qquad
  a_{\Lop}(17)=859,
  \qquad
  a_{\Lop}(19)=1076,
\]
and
\[
  964\le a_{\Lop}(18)\le965,
  \qquad
  1193\le a_{\Lop}(20)\le1194.
\]
\end{proposition}

\begin{proof}
If \(|D|\le t_3(n)-2\), then
\[
  \sigma(D,E)\le2(t_3(n)-2)+t_4(n),
\]
which saves four crossings.

Now suppose \(|D|=t_3(n)-1\).  If \(|E|\le t_4(n)-1\) or
\(|D\cap E|\le |D|-1\), then
\[
  \sigma(D,E)\le2t_3(n)+t_4(n)-3,
\]
which is enough.  The only dangerous subcase is
\[
  |E|=t_4(n),
  \qquad
  D\subseteq E.
\]
Then \(E=T_4(n)\), and one must delete exactly
\[
  t_4(n)-(t_3(n)-1)=(t_4(n)-t_3(n))+1
\]
edges from \(T_4(n)\) while destroying all transversal \(K_4\)'s.  Lemma
\ref{lem:k4-cover} rules this out in all required cases:
\[
\begin{array}{@{}c c c c@{}}
\toprule
n&T_4(n)\text{ parts}&(t_4-t_3)+1&\text{cover lower bound}\\
\midrule
15&4,4,4,3&10&12\\
17&5,4,4,4&13&16\\
18&5,5,4,4&14&16\\
19&5,5,5,4&16&20\\
20&5,5,5,5&18&25\\
\bottomrule
\end{array}
\]

It remains to take \(D=T_3(n)\).  With \(\delta=t_4(n)-t_3(n)\), saving three
crossings is equivalent, by \eqref{eq:sigma-tripartite}, to
\[
  x-2m\le\delta-3.
\]
If \(m\ge3\), then \(|E|\le t_4(n)\) gives
\(x-2m\le\delta-m\le\delta-3\).  For \(m=0,1,2\), Lemma
\ref{lem:internal-bounds} gives:
\[
\begin{array}{@{}c c c c c c@{}}
\toprule
n&(p,q,r)&\delta&m=0&m=1&m=2\\
\midrule
15&(5,5,5)&9&6&8&10\\
17&(6,6,5)&12&9&10&13\\
18&(6,6,6)&13&9&10&13\\
19&(7,6,6)&15&12&12&15\\
20&(7,7,6)&17&12&12&16\\
\bottomrule
\end{array}
\]
Every row satisfies \(x-2m\le\delta-3\).  The lower construction matches the
resulting upper bound for \(n=15,17,19\), and leaves one crossing open for
\(n=18,20\).
\end{proof}

\section{Numerical table}

Combining Paulsen's exact cases \(n=5,6,7,10\) with Propositions
\ref{prop:one-defect}, \ref{prop:two-defect}, and \ref{prop:three-defect}
gives the following table.  The final column is the consequence of this note,
not the current OEIS status.

\begin{table}[htbp]
\centering
\begin{tabular}{@{}c r r r l@{}}
\toprule
\(n\)&construction&refined upper&Paulsen upper&conclusion\\
\midrule
0&1&1&1&\(a_{\Lop}(0)=1\)\\
1&2&2&2&\(a_{\Lop}(1)=2\)\\
2&10&10&10&\(a_{\Lop}(2)=10\)\\
3&25&25&25&\(a_{\Lop}(3)=25\)\\
4&45&45&45&\(a_{\Lop}(4)=45\)\\
5&71&71&71&\(a_{\Lop}(5)=71\)\\
6&104&104&104&\(a_{\Lop}(6)=104\)\\
7&142&142&142&\(a_{\Lop}(7)=142\)\\
8&186&186&187&\(a_{\Lop}(8)=186\)\\
9&237&237&238&\(a_{\Lop}(9)=237\)\\
10&294&294&294&\(a_{\Lop}(10)=294\)\\
11&356&356&357&\(a_{\Lop}(11)=356\)\\
12&425&425&427&\(a_{\Lop}(12)=425\)\\
13&500&500&501&\(a_{\Lop}(13)=500\)\\
14&580&580&582&\(a_{\Lop}(14)=580\)\\
15&667&667&670&\(a_{\Lop}(15)=667\)\\
16&761&761&763&\(a_{\Lop}(16)=761\)\\
17&859&859&862&\(a_{\Lop}(17)=859\)\\
18&964&965&968&\(964\le a_{\Lop}(18)\le965\)\\
19&1076&1076&1079&\(a_{\Lop}(19)=1076\)\\
20&1193&1194&1197&\(1193\le a_{\Lop}(20)\le1194\)\\
\bottomrule
\end{tabular}
\caption{Refined bounds through \(n=20\).  Paulsen's upper bound is
\(U_P(n)+n+1\).}
\label{tab:final-bounds}
\end{table}

\begin{remark}[what remains at \(n=18\) and \(n=20\)]
For \(n=18\) and \(n=20\), the lower construction is still one crossing below
the refined upper bound.  Proving exactness in either case requires saving one
more unit in \(\sigma(D,E)\).  In particular, for \(n=18\), the case
\(D=T_3(18)\) with parts \((6,6,6)\) is not the apparent obstruction: an
exhaustive check over the missing cross-edge types gives \(x\le15\) when
\(m=3\), which would already save four crossings in that subcase.  The
remaining difficulty lies in near-extremal cases with \(|D|=t_3(18)-1\), or
analogous near-extremal cases at \(n=20\).
\end{remark}

\section*{Acknowledgments}

We thank David O. H. Cutler, Jonas Karlsson, Neil J. A. Sloane, and Matthias
Paulsen for the lollipop problem and for the geometric ideas on which this
note builds.  OpenAI's ChatGPT was used during exploratory algebra checking,
finite-case audit organization, and drafting assistance.

\clearpage
\appendix

\section{Paulsen's forced intriguing pair}\label{app:paulsen-proof}

We recall Paulsen's proof that among five circles, some pair is intriguing.
Suppose, toward a contradiction, that five circles with centers
\(x_i\in\mathbb R^2\) and radii \(r_i>0\) are pairwise non-intriguing.  Thus
every pair intersects at an obtuse angle, equivalently
\begin{equation}\label{eq:obtuse-circle-condition}
  r_i^2+r_j^2 < |x_i-x_j|^2 < (r_i+r_j)^2
  \qquad (i\ne j).
\end{equation}
For each circle set
\[
  v_i=\frac1{r_i}\bigl(1,\, r_i^2-|x_i|^2,\, x_i\bigr)\in\mathbb R^4,
\]
and write vectors in \(\mathbb R^4\) as triples \((\alpha,\beta,x)\), with
\(x\in\mathbb R^2\).  Endow \(\mathbb R^4\) with the symmetric bilinear form
\[
  \langle(\alpha,\beta,x),(\alpha',\beta',x')\rangle
  =\frac12\alpha\beta'+\frac12\beta\alpha'+x\cdot x'.
\]
A direct computation gives
\[
  \langle v_i,v_i\rangle=1,
  \qquad
  -1<\langle v_i,v_j\rangle<0\quad(i\ne j).
\]
Since five vectors in \(\mathbb R^4\) are linearly dependent, choose a nonzero
linear relation and move the negative terms to the other side:
\[
  \sum_{i\in P} a_i v_i = \sum_{j\in N} b_j v_j,
  \qquad
  a_i,b_j>0,
\]
where \(P\) and \(N\) are nonempty disjoint sets.  Replacing the relation by
its negative if necessary, assume \(|P|\le2\).  Let
\(w=\sum_{i\in P}v_i\).  Then
\[
  \left\langle \sum_{i\in P}a_i v_i,w\right\rangle>0,
\]
because \(|P|=1\) is immediate and, if \(P=\{1,2\}\), the expression is
\((a_1+a_2)(1+\langle v_1,v_2\rangle)>0\).  On the other hand,
\[
  \left\langle \sum_{j\in N}b_j v_j,w\right\rangle<0,
\]
since every cross inner product \(\langle v_j,v_i\rangle\), with
\(j\in N\) and \(i\in P\), is negative.  This contradicts the linear relation.
Therefore five pairwise non-intriguing circles cannot exist.

\section{The lower construction}\label{app:lower-table}

The standard lower construction blows up the optimal four-lollipop crossing
table
\[
\begin{array}{c|cccc}
 &1&2&3&4\\
\hline
1&*&5&7&7\\
2&&*&7&7\\
3&&&*&7\\
4&&&&*
\end{array}
\]
by replacing lollipop \(i\) with \(m_i\) nearby perturbed copies.  Crossings
between different clusters are inherited from the table, while pairs inside
one cluster contribute \(4\) crossings.  Hence
\begin{align*}
  L(m_1,m_2,m_3,m_4)
  ={}&5m_1m_2
  +7m_1m_3+7m_1m_4+7m_2m_3+7m_2m_4+7m_3m_4\\
  &+4\sum_{i=1}^4\binom{m_i}{2}.
\end{align*}
The lower crossing bound is
\[
  L(n)=\max_{m_1+m_2+m_3+m_4=n}L(m_1,m_2,m_3,m_4),
\]
and the corresponding lower region bound is \(L(n)+n+1\).

\begin{center}
\begin{tabular}{@{}c c r r@{}}
\toprule
\(n\)&one optimizing \((m_1,m_2,m_3,m_4)\)&\(L(n)\)&\(L(n)+n+1\)\\
\midrule
4&(1,1,1,1)&40&45\\
5&(1,1,1,2)&65&71\\
6&(1,1,2,2)&97&104\\
7&(1,2,2,2)&134&142\\
8&(1,2,2,3)&177&186\\
9&(1,2,3,3)&227&237\\
10&(2,2,3,3)&283&294\\
11&(2,2,3,4)&344&356\\
12&(2,2,4,4)&412&425\\
13&(2,3,4,4)&486&500\\
14&(2,3,4,5)&565&580\\
15&(2,3,5,5)&651&667\\
16&(3,3,5,5)&744&761\\
17&(3,3,5,6)&841&859\\
18&(3,3,6,6)&945&964\\
19&(3,4,6,6)&1056&1076\\
20&(4,4,6,6)&1172&1193\\
\bottomrule
\end{tabular}
\end{center}

\end{document}